\documentclass[11pt,leqno]{amsart}
\usepackage{amsmath,amssymb,amsthm}
\usepackage{hyperref}
\newtheorem{theorem}{Theorem}[section]
\newtheorem{lemma}[theorem]{Lemma}

\newtheorem{proposition}[theorem]{Proposition}
\newtheorem{corollary}[theorem]{Corollary}
\theoremstyle{definition}

\newtheorem{pro}[theorem]{Question}
\theoremstyle{remark}
\newtheorem{remark}[theorem]{Remark}
\numberwithin{equation}{section}

\def\fnote#1{\footnote}

\def\ignora#1{}
\def\n3#1{\left\vert  \! \left\vert \! \left\vert \, #1 \, \right\vert \!
  \right\vert \! \right\vert }

\begin{document}

\title{ Some results on almost square Banach spaces }

\author{Julio Becerra Guerrero, Gin{\'e}s L{\'o}pez-P{\'e}rez and Abraham Rueda Zoca}
\address{Universidad de Granada, Facultad de Ciencias.
Departamento de An\'{a}lisis Matem\'{a}tico, 18071-Granada
(Spain)} \email{glopezp@ugr.es, juliobg@ugr.es, arz0001@correo.ugr.es}

\maketitle\markboth{J. Becerra, G. L\'{o}pez and A. Rueda}{Some results on almost square Banach spaces}

\begin{abstract}
We study almost square Banach spaces under a topological point of view. Indeed, we prove that the class of Banach spaces which admits an equivalent norm to be ASQ is that of  those Banach spaces which contain an isomorphic copy of $c_0$. We also prove that the symmetric projective tensor products of an almost square Banach space have the strong diameter two property.
\end{abstract}

\section{Introduction}

\bigskip
\par

The study of the size of  slices,  non-empty relatively weakly open and convex combination of slices of the unit ball of a Banach space has emerged in the last few years. A Banach space $X$ is said to have the slice diameter two property (respectively diameter two property, strong diameter two property) if every slice (respectively non-empty relatively weakly open subset, convex combination of slices of the unit ball) has diameter two. Such properties, which have been proved to be different in an extreme way \cite{blr2}, have show to have strong links with other properties of the geometry of a Banach space such as having octahedral norms \cite{blr4}. Under this frame, a new class of Banach spaces have recently appeared: the so-called \textit{almost square Banach spaces}.

 According to \cite{all}, a Banach space  $X$ is said to be 
\begin{enumerate}
\item \textit{locally almost square (LASQ)} if for every $x\in S_X$ there exists a sequence $\{y_n\}$ in $B_X$ such that $\Vert x\pm y_n\Vert\rightarrow 1$ and $\Vert y_n\Vert\rightarrow 1$.

\item \textit{weakly almost square (WASQ)} if for every $x\in S_X$ there exists a sequence $\{y_n\}$ in $B_X$ such that $\Vert x\pm y_n\Vert\rightarrow 1$, $\Vert y_n\Vert\rightarrow 1$ and $\{y_n\}\rightarrow 0$ weakly.

\item \textit{almost square (ASQ)} if for every $x_1,\ldots, x_n$ elements of $S_X$ there exists a sequence $\{y_n\}$ in $S_X$ such that $\Vert y_n\Vert\rightarrow 1$ and $\Vert x_i\pm y_n\Vert\rightarrow 1$ for every $i\in\{1,\ldots, n\}$.

\end{enumerate} 

On the one hand it is obvious that WASQ Banach spaces are LASQ Banach spaces, and it is also known that ASQ Banach spaces are WASQ Banach spaces \cite[Theorem 2.8]{all}. On the other hand there are several examples of Banach spaces which are ASQ as $c_0$ \cite[Example 3.1]{all}, the Hagler space $JH$ \cite[Lemma 3.1]{blr}, non-reflexive $M$-embedded Banach spaces \cite[Corollary 4.3]{all}...

In \cite{all} it is pointed out the nice relation between almost square Banach spaces and diameter two properties. Indeed, LASQ (respectively WASQ, ASQ) Banach spaces enjoy to have the slice diameter two property (respectively the diameter two property, strong diameter two property). However, in such paper not only do the autors study almost square Banach spaces from a geometrical point of view but also from an isomorphic one.

Indeed, every ASQ Banach space contains an isomorphic copy of $c_0$ \cite[Lemma 2.6]{all}. In addition, as the property of being an ASQ Banach space is preserved under taking $\ell_\infty$ sums with any other Banach space, Banach spaces containing a complemented copy of $c_0$ can be equivalently renormed to be ASQ. Consequently, a separable Banach space $X$ can be equivalently renormed to be ASQ if, and only if, $X$ contains an isomorphic copy of $c_0$ \cite[Corollary 2.10]{all}. Moreover, the authors ask whether the hypothesis of separability can be eliminated.

The main aim of this note is to provide a positive answer to the above question, proving that every Banach space containing an isomorphic copy of $c_0$ can be equivalently renormed to be ASQ. To this aim, in section 2, we will firstly renorm the Banach space $\ell_\infty$ to be ASQ. Then we will prove that every Banach space containig an isomorphic copy of $c_0$ can be equivalently renormed to be ASQ by giving a suitable renorming of the bidual space. Section 3 will be devoted to analyse the the relations between ASQ Banach spaces and the strong diameter two property in symmetric projective tensor products. In the last few years several papers have appeared related to diameter two properties in tensor products spaces (see \cite{ab,abr,blr3}) and, even though it has been proved stability results of the slice diameter two property and strong diameter two property in projective tensor products of Banach spaces (see \cite{blr3}), such results seem to be unknown for symmetric tensor products. So, in Section 3, we shall prove that the symmetric projective tensor products of ASQ Banach spaces have the strong diameter two property. Finally, Section 4 we will be devoted to exhibit open problems related to ASQ spaces.

We shall now introduce some notation. We will consider real Banach
spaces. $B_X$, respectively $S_X$,  stands for the closed unit ball, respectively the
unit sphere, of the Banach space $X$. We denote by $X^*$ the
topological dual space of $X$. Given $I$ a non-empty set and $\mathcal U$ an ultrafilter on $I$, recall that $\mathcal U$ is said to be \textit{principal} if there exists $i\in I$ such that $\mathcal U:=\{Y\subseteq I\ /\ i\in Y\}$. In other case, $\mathcal U$ is said to be \textit{non-principal} (see \cite{wil} for background). Given $f:I\longrightarrow \mathbb R$ a bounded function we will denote by
$$\lim\limits_{\mathcal U} f$$
the limit of $f$ under the ultrafilter $\mathcal U$. It is well known that in the particular case $I=\mathbb N$, there are non-principal ultrafilters. Moreover, given $\mathcal U$ a non-principal ultrafilter it follows that
$$\lim\limits_{\mathcal U} x=\lim\limits_{n\rightarrow \infty} x(n)$$
for every convergent sequence $x$.

Finally, according to \cite{ab}, for a Banach space $X$ and $N\in\mathbb N$, we will denote by $\widehat{\otimes}_{\pi,s,N}X$ the \textit{symmetric projective N-tensor product of $X$}. This space is the completion of the linear space generated by $\left\{x^N:=\mathop{x\otimes\ldots\otimes x}\limits^{N}: x\in X\right\} $ under the norm given by
$$\Vert z\Vert:=\inf \left\{\sum_{i=1}^k\vert\lambda_i\vert  :z=\sum_{i=1}^k \lambda_i x_i^N, \lambda_i\in\mathbb R, x_i\in S_X \ \forall i\in\{1,\ldots, k\}\right\}.$$
It is well known that its topological dual space is identified with the space of all $N$-homogeneous and bounded polynomials on $X$ by the action
$$P\left(x^N\right):=P(x)\ \ \forall x\in X$$
for each $N$-homogeneous and bounded polynomial $P$ (see \cite{flo} for background). We will denote by $\mathcal P(^N X,Y)$ the space of $N$-homogeneous polynomials from $X$ to $Y$.

\section{A renorming theorem for ASQ Banach spaces}\label{secisomorfos}

It is clear that each Banach space containing a complemented copy of $c_0$ can be equivalently renormed to be ASQ. Since $c_0$ is not complemented in $\ell_\infty$ \cite[Theorem 5.15]{rusos} it seems natural to wonder whether $\ell_\infty$, which is not an ASQ space, can be equivalently renormed to be ASQ.

In order to exhibit such a norm on $\ell_\infty$, consider $\mathcal U$ a non-principal ultrafilter on $\mathbb N$ and define, given $x\in \ell_\infty$, the following
$$\lim(x):=\lim_\mathcal U x.$$
Then $\lim:\ell_\infty\longrightarrow \mathbb R$ is linear and continuous. In fact, it is easy to prove that $\Vert \lim\Vert=1$. 

Now we are ready to prove the following

\begin{theorem}\label{renormainfi}
There exists an equivalent norm on $\ell_\infty$, say $|||\cdot|||$, such that the Banach space $(\ell_\infty,|||\cdot|||)$ is an ASQ Banach space.
\end{theorem}

\begin{proof}
Consider on $\ell_\infty$ the norm given by
$$|||x|||:=\max\left\{\vert \lim(x)\vert,\sup_n\vert x(n)-\lim(x)\vert\right\},$$
Let us prove that the norm defined above is equivalent to the classical one on $\ell_\infty$. To this aim consider $x\in \ell_\infty$. Now, on the one hand
$$||| x|||\leq \max\left\{\vert \lim(x)\vert,\sup_n \vert x(n)\vert+\vert \lim(x)\vert\right\}\leq \Vert x\Vert_\infty+\Vert x\Vert_\infty=2\Vert x\Vert_\infty,$$
as $\Vert \lim\Vert=1$.
On the other hand
$$|||x|||\geq \sup_n \vert x(n)-\lim(x)\vert\geq \sup_n\vert x(n)\vert-\vert\lim(x)\vert.$$
Now
$$\Vert x\Vert_\infty\leq ||| x|||+\vert\lim(x)\vert\leq \vert \lim(x)\vert+\sup_n\vert x(n)-\lim(x)\vert+\vert \lim(x)\vert$$
$$\leq 2(\vert \lim(x)\vert+\sup_n\vert x(n)-\lim(x)\vert)\leq 4 \max\left\{\vert\lim(x)\vert,\sup_n\vert x(n)-\lim(x)\vert\right\}=4|||x|||. $$
So $|||\cdot|||$ and $\Vert\cdot\Vert_\infty$ are equivalent norms. Let us now prove that $(\ell_\infty,|||\cdot|||)$ is an ASQ Banach space. To this aim pick $x_1,\ldots, x_n\in S_{\ell_\infty}$ and $\varepsilon>0$.
Given $i\in\{1,\ldots, n\}$ consider the sets
$$A_i:=\{n\in\mathbb N\ /\ \vert x_i(n)-\lim x_i\vert<\varepsilon\}.$$
Then $A_1,\ldots, A_n\in\mathcal U$ by the definition of the limit by ultrafilter, so $A:=\bigcap\limits_{i=1}^n A_i\in\mathcal U$ as a finite intersection of elements of $\mathcal U$. Since $\mathcal U$ is an ultrafilter we have that $A\neq \emptyset$, so pick $n\in A$. Now let us estimate $|||x_i\pm e_n|||\ \forall i\in\{1,\ldots, n\}$. To this aim pick $i\in\{1,\ldots, n\}$. Then, on the one hand,
$$\vert\lim  (x_i\pm e_n)\vert=\vert\lim  (x_i)\vert$$
since $\lim e_n=0$.
On the other hand
$$\sup_k \vert x_i(k)\pm e_n(k)-\lim (x_i\pm e_n)\vert$$
$$=\max\left\{\sup_{k\neq n}\vert x_i(k)-\lim(x_i)\vert,\vert x_i(n)-\lim(x_i)\pm 1\vert \right\}$$
$$\leq \max\left\{\sup_{k\neq n}\vert x_i(k)-\lim(x_i)\vert,\vert x_i(n)-\lim(x_i)\vert+1\right\}$$
$$\leq\max\left\{\sup_n\vert x_i(n)-\lim(x_i)\vert,1+\varepsilon\right\}.$$
Consequently, by definition of the norm $|||\cdot|||$, one has $$|||x_i\pm e_n|||\leq \max\{|||x_i|||,1+\varepsilon\}=1+\varepsilon.$$
Moreover, check that
$$\vert \lim(e_n)\vert=0$$ and $$\sup_k\vert e_n(k)-\lim(e_k)\vert=1.$$
Thus $e_n\in S_{\ell_\infty}$. From \cite[Proposition 2.1]{all} we get that $(\ell_\infty,|||\cdot|||)$ is ASQ Banach space, as desired.
\end{proof}

\begin{remark}\label{observaclave}
From the above proof it follows that, given $x_1,\ldots, x_n\in S_{\ell_\infty}$ and $\varepsilon>0$ we can find $y\in S_{c_0}$ such that $|||x_i\pm y|||\leq 1+\varepsilon$. Roughly speaking we can say that the fact that $\ell_\infty$ under the norm of above Theorem is ASQ relies on the subspace $c_0$. This simple observation will be the key of the general renorming result.
\end{remark}

As we have pointed out above, the Banach space $\ell_\infty$ plays an important role as example of Banach space containing an isomorphic copy of $c_0$ which can be equivalently renormed to be ASQ. However, as dual Banach spaces containing an isomorphic copy of $c_0$ actually contain a complemented copy of $\ell_\infty$ \cite[Proposition 2.e.8]{litza}, we can deduce our general result from this particular example by giving a suitable renorming in the bidual space.

\begin{theorem}\label{renormageneral}

Let $X$ be a Banach space containing an isomorphic copy of $c_0$. Then there exists an equivalent norm on $X$ such that $X$ is an ASQ space under the new norm.

\end{theorem}

\begin{proof}

Assume that $X$ contains a subspace $Y$ which is isometric to $c_0$.

As $Y^{**}\subseteq X^{**}$ is linearly isometric to $\ell_\infty$, then $Y^{**}$ is complemented in  $X^{**}$ \cite[Proposition 5.13]{rusos}.

Then we can consider on $X^{**}$ an equivalent norm so that
$$X^{**}=Y^{**}\oplus_\infty Z,$$
and such norm agrees with the original one of $Y^{**}$.

Now we can consider on $Y^{**}$ the norm defined in Theorem \ref{renormainfi},  so  $Y^{**}$ becomes into an ASQ space and agrees with the original norm on $Y\subseteq X$. This defines an equivalent norm on $X^{**}$ which we will denote by $\Vert\cdot\Vert$. Clearly $X^{**}$ is an ASQ space \cite[Proposition 5.7]{all}. Our aim is to prove that $X$ is an ASQ space following similar ideas to \cite[Proposition 5.7]{all} and Remark \ref{observaclave}.

To this aim pick $x_1,\ldots, x_n\in S_X$ and $\varepsilon>0$. Now $x_i\in X^{**}=Z\oplus_\infty Y^{**}$ for each $i\in\{1,\ldots, n\}$, so we can find $z_i\in Z$ and $y_i\in Y^{**}$ such that $x_i=(z_i,y_i)\ \forall i\in\{1,\ldots, n\}$. We can assume, making a perturbation argument if necessary, that $y_i\neq 0 \ \forall i\in\{1,\ldots, n\}$. From Remark \ref{observaclave} we can find $y\in S_{c_0}$ such that
\begin{equation}\label{teogeneaplicaASQ}
\left\Vert \frac{y_i}{\Vert y_i\Vert}\pm y \right\Vert\leq 1+\varepsilon.
\end{equation}
Define $z:=(0,y)\in S_{c_0}\subseteq X$. Then
$$\Vert x_i\pm z\Vert=\max\{\Vert z_i\Vert,\Vert y_i\pm y\Vert\}\leq \max\{1,\Vert y_i\pm y\Vert\}$$
$$=\left\{1,\left\Vert \Vert y_i\Vert\left(\frac{y_i}{\Vert y_i\Vert}\pm y\right)\pm(1-\Vert y_i\Vert)y \right\Vert  \right\}$$
$$\mathop{\leq}
\limits^{\mbox{(\ref{teogeneaplicaASQ})}} \max\left\{1,\Vert y_i\Vert(1+\varepsilon)+(1-\Vert y_i\Vert)\Vert y\Vert \right\}\leq 1+\varepsilon.$$
To sum up we have proved that given $x_1,\ldots, x_n\in S_X$ and $\varepsilon>0$ we can find $z\in S_X$ such that 
$$\Vert x_i\pm z\Vert\leq 1+\varepsilon.$$
Thus $X$ is an ASQ space under the new equivalent norm, so we are done.

\end{proof}

Above Theorem allows us to strengthen \cite[Proposition 4.7]{aln}, where it is proved that every Banach space containing an isomorphic copy of $c_0$ can be equivalently renormed to have the strong diameter two property.

Moreover, from Theorem \ref{renormageneral} and \cite[Theorem 2.4]{all} we get the following

\begin{corollary}\label{carac0asq}
Let $X$ be a Banach space. Then there exists an equivalent norm on $X$ such that $X$ is an ASQ Banach space under the new norm  if, and only if, $X$ contains an isomorphic copy of $c_0$.
\end{corollary}

In \cite{all} the relation between ASQ Banach spaces and the intersection property is pointed out. Recall that a Banach space $X$ has the \textit{intersection property} if for every $\varepsilon>0$ there exist  $x_1,\ldots, x_n\in X$ such that $\Vert x_i\Vert<1$ and such that if $y\in X$ verifies that $\Vert x_i-y\Vert\leq 1$ for every $i\in\{1,\ldots, n\}$ then $\Vert y\Vert\leq \varepsilon$. Given $0<\varepsilon<1$, $X$ is said to \textit{$\varepsilon$-fail the intersection property} if $\gamma(\varepsilon)=1$, where 
$$\gamma(\varepsilon):=\sup_{x_1,\ldots, x_n\in B_{[0,1)}}\inf_{y\in B_{(\varepsilon,1]}} \max_{1\leq i\leq n}\Vert x_i-y\Vert\ \forall\ 0<\varepsilon<1,$$
and $B_I:=\{x\in X\ /\ \Vert x\Vert\in I\}$ for each $I\subseteq \mathbb R^+$. Finally, a Banach space is said to \textit{fail the intersection property} if $X$ $\varepsilon$-fails the intersection property for some $0<\varepsilon<1$.

On the one hand, it is known that a Banach space $X$ is ASQ if, and only if, $X$ $\varepsilon$-fails the intersection property for every $0<\varepsilon<1$ \cite[Proposition 6.1]{all}. On the other hand, it is known that a Banach space admits an equivalent norm which fails the intersection property if, and only if, $X$ contains an isomorphic copy of $c_0$ \cite[Theorem 1.7]{hr}. Now we can improve  above Theorem as an straightforward application of Corollary \ref{carac0asq}.

\begin{theorem}\label{renormaip}

Let $X$ be a Banach space. Then $X$ admits an equivalent norm which $\varepsilon$-fails the intersection property for each $0<\varepsilon<1$ if, and only if, $X$ contains an isomorphic copy of $c_0$.

\end{theorem}

\section{ASQ Banach spaces and symmetric tensor products}\label{sectensores}

One of the most important fact when one tries to prove that ASQ Banach spaces have the strong diameter two property is that in such spaces we have a lot of weakly-null sequences which are equivalent to the $c_0$ basis (see \cite[Lemma 2.6]{all}). So, in order to prove that the symmetric projective tensor products of an ASQ Banach space have the strong diameter two property, we shall begin by proving the following Lemma, which asserts that such sequences are still weakly-null when they are considered in the tensor space.

\begin{lemma}\label{sucedebinulatensosime}

Let $X$ be a Banach space and consider $\{y_n\}\subseteq S_X$ a sequence equivalent to the usual basis of $c_0$. Pick $N\in\mathbb N$. Then $\{e_n^N\}\rightarrow 0$ in the weak topology of $Y:=\widehat{\otimes}_{\pi,s,N} X$.

\end{lemma}

\begin{proof}
Pick $P$ a $N$-homogenous polynomial on $X$. Define $Y:=\overline{span}\{e_n:n\in\mathbb N\}$, which is a subspace of $X$ which is isomorphic to $c_0$. Consider $Q=P_{|Y}$, which is a polynomial in $Y$. In fact, if $P(x)=M(x,\ldots, x)$ for suitable $N$-lineal form $M$, then
$$Q(y)=M_{|Y^N}(y,\ldots, y)\ \forall y\in Y.$$
For each $n\in\mathbb N$ it follows that
$$P(e_n)=Q(e_n).$$
Moreover, as each polynomial on $Y$ is weakly sequentially continuous (because $Y$ is linearly isomorphic to $c_0$ and such space has the polynominal Dunford-Pettis property \cite{jpz}) we conclude that $\{Q(e_n)\}\rightarrow 0$, so we are done.
\end{proof}

Note that given a Banach space $X$ with the strong diameter two property we can find, in every convex combination of slices of its unit ball, elements whose norm is as close to 1 as desired (see \cite[Lemma 2.1]{blr3}). In order to prove the announced result we shall before verify that such property is satisfied by every symmetric projective tensor product of an ASQ Banach space.

\begin{lemma}\label{combiconvextensosim}

Let $X$ be an ASQ space and $N\in\mathbb N$. Consider $Y:=\widehat{\otimes}_{\pi,s,N} X$ and $C:=\sum_{i=1}^k\lambda_i S(B_Y,P_i,\alpha_i)$ a convex combination of slices of $B_Y$. Then, for each $\varepsilon>0$, we can find $f\in S_{X^*}$ and $y_i^N\in S_i$  such that
$$f(y_i)>1-\varepsilon\ \forall i\in\{1,\ldots, k\}.$$
\end{lemma}

\begin{proof}
For each $i\in\{1,\ldots, k\}$ consider $x_i\in S_X$ such that $x^N_i\in S_i$, in other words, $Q_i(x_i)>1-\alpha_i$ for each $i\in\{1,\ldots, k\}$. Since $Q_i$ is an $N$-homogeneous polynomial for each $i\in\{1,\ldots, k\}$ we can find $\varepsilon_0>0$ such that
$$\label{condiepsitensime}
0<\varepsilon<\varepsilon_0\Rightarrow Q_i\left(\frac{x_i}{1+\varepsilon}\right)>1-\alpha_i\ \forall i\in\{1,\ldots, k\}.$$
Following \cite[Lemma 2.6]{all} we can find $\{y_n\}$ a weakly-null sequence in $S_X$ which is $1+\varepsilon$-isometric to the usual basis of $c_0$ such that
$$\Vert x_i\pm y_n\Vert\rightarrow 1\ \forall i\in\{1,\ldots, k\}.$$
Now, as $\{y_n^N\}\rightarrow 0$ weakly in $Y$ because of Lemma \ref{sucedebinulatensosime} we conclude that $\{(x_i\pm y_n)^N\}\rightarrow x_i$ in the weak topology of $Y$ for each $i\in\{1,\ldots, k\}$. In fact, given $i\in\{1,\ldots, k\}$, one has that $\{(x_i\pm y_n)^N\}\rightarrow x_i$ if, and only if, $\{(x_i\pm y_n-x_i)^N\}\rightarrow 0$ \cite[Lemma 1.1]{fj}.

From facts above we can find $n$ large enough to ensure
$$\frac{\Vert x_i\pm y_n\Vert}{1+\varepsilon}\leq 1\ \forall i\in\{1,\ldots, k\},$$
and
$$Q_i\left(\frac{x_i\pm y_n}{1+\varepsilon}\right)>1-\alpha_i\ \forall i\in\{1,\ldots, k\}.$$

Now, on the one hand, $\sum_{i=1}^k\lambda_i\frac{(x_i\pm y_n)^N}{1+\varepsilon}\in C$.
On the other hand, consider $f\in S_{X^*}$ such that $f(y_n)=1$. As $f(x_i\pm y_n)\leq 1+\varepsilon$ for each $i\in\{1,\ldots, k\}$ we conclude that
$$\vert f(x_i)\vert\leq \varepsilon\ \forall i\in\{1,\ldots, k\}.$$
Now, defining $y_i:=\frac{x_i+ y_n}{1+\varepsilon}$, one has
$$f(y_i)=\frac{f(x_i)+f(y_n)}{1+\varepsilon}>\frac{1-\varepsilon}{1+\varepsilon}\ \forall i\in\{1,\ldots, k\}.$$
As $0<\varepsilon<\varepsilon_0$ was arbitrary we get the desired result.
\end{proof}

Now we are ready to prove the main result of the section.

\begin{theorem}\label{ASQproyesimet}
Let $X$ be an ASQ space and $N\in\mathbb N$. Then $Y:=\widehat{\otimes}_{\pi,s,N} X$ has the strong diameter two property.
\end{theorem}

\begin{proof}

Pick $k\in\mathbb N$ and $S_i:=S(B_Y,Q_i,\alpha_i)$ slices of $B_Y$, where $Q_1,\ldots, Q_k$ are norm-one $N$-homogeneous polynomials on $X$ and pick $\lambda_1,\ldots, \lambda_k\in [0,1]$ such that $\sum_{i=1}^k\lambda_i=1$. Define $C:=\sum_{i=1}^k \lambda_i S_i$ a convex combination of slices. Our aim is to prove that $diam(C)=2$. To this aim pick $\varepsilon>0$. In view of above lemma we can find $x_i^N\in S_i$ for each $i\in\{1,\ldots, k\}$  and $g\in S_{X^*}$ such that
$$g(x_i)>1-\varepsilon\ \forall i\in\{1,\ldots, k\}.$$
Now, again from computations of above Lemma, we can find $y\in S_X$ such that
$$\frac{(x_i\pm y)^N}{1+\varepsilon}\in S_i\ \forall i\in\{1,\ldots, k\}.$$ Now, on the one hand, $\sum_{i=1}^k \lambda_i \frac{(x_i\pm y)^N}{1+\varepsilon}\in C.$
On the other hand, let us estimate
$$diam(C)\geq \left\Vert\sum_{i=1}^k \lambda_i  \frac{(x_i+ y)^N-(x_i-y)^N}{1+\varepsilon}\right\Vert.$$
To this aim, pick $f\in S_{X^*}$ such that $f(y)=1$. As $f(x_i\pm y)\leq 1+\varepsilon\ \forall i\in\{1,\ldots, k\}$ we conclude that 
$$\vert f(x_i)\vert \leq\varepsilon\ \forall i\in\{1,\ldots, k\}.$$
Now let us argue by cases:

\begin{itemize}
\item If $N$ is odd, define $P(z):=f(z)^N$ for each $z\in X$. Then
$$\left\Vert\sum_{i=1}^k \lambda_i  \frac{(x_i+ y)^N-(x_i-y)^N}{1+\varepsilon}\right\Vert\geq \sum_{i=1}^k\lambda_i \frac{P(x+y)-P(x-y)}{1+\varepsilon}$$
$$=\sum_{i=1}^k \lambda_i \frac{f(x_i+y)^N-f(x_i-y)^N}{1+\varepsilon}=\sum_{i=1}^k \lambda_i \frac{f(x_i+y)^N+f(y-x_i)^N}{1+\varepsilon}$$
$$\geq \frac{(1-\varepsilon)^N+(1-\varepsilon)^N}{1+\varepsilon}\sum_{i=1}^k\lambda_i
=2\frac{(1-\varepsilon)^N}{1+\varepsilon}.$$
\item If $N$ is even, define $P(z):=f(z)^{N-1}g(z)$ for each $z\in X$. Again, $\Vert P\Vert\leq 1$. Now we conclude
$$\left\Vert\sum_{i=1}^k \lambda_i \frac{(x_i+ y)^N-(x_i-y)^N}{1+\varepsilon}\right\Vert\geq
\sum_{i=1}^k\lambda_i \frac{P(x_i+y)-P(x_i-y)}{1+\varepsilon}$$
$$=\sum_{i=1}^k \lambda_i \frac{f(x_i+y)^{N-1}
g(x_i+y)-f(x_i-y)^{N-1}g(x_i-y)}{1+\varepsilon}$$
$$=\sum_{i=1}^k\lambda_i \frac{g(x_i)(f(x_i+y)^{N-1}-f(x_i-y)^{N-1})+g(y)(f(x_i+y)^{N-1}+f(x_i-y)^{N-1})}{1+\varepsilon}.$$
Now, as $N-1$ is odd, we can estimate $f(x_i+y)^{N-1}-f(x_i-y)^{N-1}$ by $2(1-\varepsilon)^{N-1}$ as in the above case for each $i\in\{1,\ldots, k\}$. On the other hand, check that $\vert g(y)\vert\leq 2\varepsilon$ as $g(x_i\pm y)\leq 1+\varepsilon$ and $g(x_i)>1-\varepsilon$  for each $i\in\{1,\ldots, k\}$. Bearing in mind that $g(x_i)>1-\varepsilon\ \forall i\in\{1,\ldots, k\}$ we conclude that
$$diam(W)\geq \frac{2(1-\varepsilon)^{N}-4
\varepsilon(1+\varepsilon)^{N-1}}{1+\varepsilon}.$$

\end{itemize}

In any case, as $0<\varepsilon<\varepsilon_0$ was arbitrary we conclude that $diam(C)=2$ as desired.
\end{proof}

\begin{remark}
Check that last above estimates are similar to the ones of \cite{ab}. Consequently, in \cite[Proposition 2.4]{ab} can be obtained the strong diameter two property under the same assumptions.

\end{remark}

Now, as an easy consequence of Theorems \ref{renormageneral} and \ref{ASQproyesimet}, we get the following

\begin{corollary}

Let $X$ be a Banach space which contains an isomorphic copy of $c_0$. Then there exists an equivalent norm on $X$ such that for each $N\in\mathbb N$ the projective symmetric tensor product $\widehat{\otimes}_{\pi,s,N}X$ has the strong diameter two property.

\end{corollary}

According to \cite{blr4}, the norm on a Banach space $X$ is said to be \textit{octahedral} if for every $\varepsilon>0$ and every $Y\subseteq X$ finite-dimensional subspace there exists $x\in S_X$ such that 
$$\Vert y+\lambda x\Vert\geq (1-\varepsilon)(\Vert y\Vert+\vert\lambda\vert)\ \forall y\in Y,\forall \lambda\in\mathbb R.$$
It is known that a Banach space $X$ an octahedral norm if, and only if, $X^*$ has the weak-star strong diameter two property \cite[Theorem 2.1]{blr4}. Consequently, from Theorem \ref{ASQproyesimet} we conclude that given an ASQ Banach space $X$ then $\mathcal P(^N X)$ has an octahedral norm for each $N\in\mathbb N$. However, we can go further bearing in mind the results of \cite{blr3}.

\begin{corollary}

Let $X$ and $Y$ be Banach spaces. If $X$ is an ASQ space and $Y$ has an octahedral norm, then $\mathcal P(^N X,Y)$ has an octahedral norm for each $N\in\mathbb N$.

\end{corollary}

\begin{proof}
As $\mathcal P(^N X,Y)$ and $L(\widehat{\otimes}_{\pi,s,N} X,Y)$ are linearly isometric \cite{flo} for each $N\in\mathbb N$, the Corollary follows from Theorem \ref{ASQproyesimet}, \cite[Theorem 2.1]{blr4} and \cite[Theorem 2.5]{blr3}.
\end{proof}

\section{Some remarks and open questions}

In \cite{all} it is posed as an open question whether there exists a dual Banach space which is ASQ. In view of Theorem \ref{renormainfi} it seems natural take advantage of Theorem \ref{renormageneral} for dual Banach spaces. Unfortunately, the technique exposed in such result does not respect the duality of Banach spaces. Indeed, we have the following

\begin{proposition}

Let $|||\cdot|||$ be the ASQ norm on $\ell_\infty$.  Then $Ext(B_{\ell_\infty})=\emptyset$. As a consequence, $(\ell_\infty,|||\cdot|||)$ is not isometric to any dual Banach space.

\end{proposition}

\begin{proof}
Consider $x\in S_{\ell_\infty}$. Then we have the following considerations:

\begin{enumerate}
\item If $\vert \lim(x)\vert<1$ then there exists $\varepsilon>0$ such that $\vert \lim(x)\pm \varepsilon\vert<1$. Now consider
$$y:=x+\varepsilon{\bf{1}}\ \ \ z:=x-\varepsilon{\bf{1}}.$$
Then clearly $\vert \lim(y)\vert\leq 1 $ and $\vert\lim(z)\vert\leq 1$. On the other hand, given $n\in\mathbb N$ one has
$$\vert y(n)-\lim(y)\vert=\vert y_n+\varepsilon-\lim(x)-\varepsilon\vert=\vert x(n)-\lim(x)\vert\leq |||x|||\leq 1.$$
So $y\in B_{\ell_\infty}$. By a similar argument we have that $z\in B_{\ell_\infty}$. As $x=\frac{y+z}{2}$ we get that $x\notin Ext(B_{\ell_\infty})$ in this case.

\item If $\vert \lim(x)\vert=1$, we shall assume with no loss of generality that $\lim(x)=1$. Then $\sup_n \vert x(n)-1\vert\leq 1$ from where we conclude that $x(n)\geq 0\ \forall n\in\mathbb N$. Moreover as $\lim(x)=1$ then we can find $\varepsilon>0$ and $n\in\mathbb N$ such that $x(n)\geq \varepsilon$. We claim that
$$x\pm \varepsilon e_n\in B_{\ell_\infty}.$$
Indeed,
$$\lim(x\pm \varepsilon e_n)=\lim(x).$$
Moreover, given $k\in\mathbb N$, 
$$\vert x(k)\pm \varepsilon e_n(k)-\lim(x\pm \varepsilon e_n)\vert=\vert x(k) \pm \varepsilon\delta_{kn}-\lim(x)\vert$$
If $k\neq n$ clearly last quantity is less than or equal to $|||x|||\leq 1$. Moreover, if $k=n$ then we have
$$\vert x(n)\pm \varepsilon -\lim(x)\vert=\vert x(n)\pm \varepsilon-1\vert\leq \vert x(n)-1\vert+\varepsilon=1-x(n)+\varepsilon\leq 1. $$
Thus $x\pm\varepsilon e_n\in B_{\ell_\infty}$, so $x\notin Ext(B_{\ell_\infty})$.

\end{enumerate}

This proves that $Ext(B_{\ell_\infty})=\emptyset$. Now $\ell_\infty$ is not a dual Banach space as an easy consequence of Krein-Milman theorem.

\end{proof}

From above Proposition we get that we can not give a dual renorming of the bidual of a Banach space containing an isomorphic copy of $c_0$ using the ideas of Theorem \ref{renormageneral}. Indeed, given $X$ a Banach space containing an isomorphic copy f $c_0$ we have considered a renorming of $X^{**}$ such that
$$X^{**}=\ell_\infty\oplus_\infty Z,$$
where we consider on $\ell_\infty$ the renorming of Theorem \ref{renormainfi}. By above Proposition the unit ball of  $\ell_\infty$ does not have any extreme point. Consequently, so does the unit ball of $X^{**}$ and, again by Krein-Milman theorem, $X^{**}$ can not be isometric to any dual Banach space. So it remains open wheter there exists a dual Banach space which is ASQ. It is even open wether there exists a dual Banach space failing the intersection property \cite[Section 4]{bh}

It is also posed as an open question in \cite{all}  if there exists any LASQ Banach space which fails to be WASQ. In this direction recall that it is proved that every Banach space which contains an isomorphic copy of $c_0$ can be equivalently renormed to have the slice diameter two property and whose new unit ball contains non-empty relatively weakly open subset whose diameter is as small as desired \cite[Theorem 2.4]{blr1}. Similarly, it is proved that every Banach space which contains an isomorphic copy of $c_0$ can be equivalently renormed to have the diameter two property and whose new unit ball contains convex combinations of slices  whose diameter is as small as desired \cite[Theorem 2.5]{blr2}. Thus we can go further and pose the following:

\begin{pro}

Let $X$ be a Banach space containing an isomorphic copy of $c_0$.

\begin{enumerate}
\item Is there an equivalent norm on $X$ which is LASQ and fails to be WASQ (or even its new unit ball contains non-empty relatively weakly open subsets whose diameter is as small as desired)?

\item Is there an equivalent norm on $X$ which is WASQ and fails to be ASQ (or even its new unit ball contains convex combinations of slices whose diameter is as small as desired)?
\end{enumerate}

\end{pro}

Finally, bearing in mind the results exposed in Section \ref{sectensores}, we will pose the following

\begin{pro}

Let $X$ be an ASQ Banach space and pick $N\in\mathbb N$. Is $\widehat{\otimes}_{\pi,s,N}X$ an ASQ Banach space?

\end{pro}

\end{document}